\documentclass{amsart}

\usepackage[pdftex,breaklinks,bookmarks]{hyperref} 

\newtheorem{theorem}{Theorem}[section]
\newtheorem{lemma}[theorem]{Lemma}
\newtheorem{proposition}[theorem]{Proposition}

\theoremstyle{definition}
\newtheorem{definition}[theorem]{Definition}
\newtheorem{example}[theorem]{Example}

\theoremstyle{remark}
\newtheorem{remark}[theorem]{Remark}

\newcommand{\IC}{\ensuremath{\mathbb{C}}}

\newcommand{\IR}{\ensuremath{\mathbb{R}}}
\newcommand{\IZ}{\ensuremath{\mathbb{Z}}}
\newcommand{\IN}{\ensuremath{\mathbb{N}}}
\newcommand{\IH}{\ensuremath{\mathbb{H}}}
\newcommand{\IP}{\ensuremath{\mathbb{P}}}

\newcommand{\cC}{\ensuremath{\mathcal{C}}}

\newcommand{\cF}{\ensuremath{\mathcal{F}}}

\newcommand{\dd}{\ensuremath{\mathrm{d}}}
\newcommand{\id}{\ensuremath{\mathbf{1}}}
\newcommand{\eps}{\ensuremath{\varepsilon}}
\newcommand{\gM}{\ensuremath{g\!\!M}}
\newcommand{\cM}{\ensuremath{c\!\!M}}
\newcommand{\EE}{\ensuremath{\operatorname{E}}}
\newcommand{\smallsetminus}{\ensuremath{\setminus}}

\newcommand{\re}[1]{\ensuremath{{\operatorname{Re}\left(#1\right)}}}
\newcommand{\im}[1]{\ensuremath{{\operatorname{Im}\left(#1\right)}}}
\newcommand{\sign}[1]{\ensuremath{{\operatorname{sign}\left(#1\right)}}}
\newcommand{\abs}[1]{\ensuremath{{\left\lvert#1\right\rvert}}} 
\renewcommand{\arg}[1]{\ensuremath{{\operatorname{arg}\left(#1\right)}}}

\newcommand{\SL}[1]{\ensuremath{{\mathrm{SL}\!\left(2, #1 \right)}}}

\newcommand{\Matrix}[4]{{\begin{pmatrix} #1 & #2 \\ #3 & #4 \end{pmatrix}}}

\newcommand{\OO}[1]{\ensuremath{\mathcal{O}\left( #1 \right)}}
\newcommand\txtfrac[2]{{\textstyle \frac{#1}{#2}}}

\newcommand{\MM}[1]{\ensuremath{\tilde{M}_{#1}}}
\newcommand{\WW}[1]{\ensuremath{\tilde{W}_{#1}}}

\title{Generalized Maass Wave Forms}

\author{T.~M\"uhlenbruch}
\address{Department of Mathematics and Computer Science, \href{http://www.fernuni-hagen.de}{FernUniversit\"at in Hagen}, 58084 Hagen, Germany}
\email{\href{mailto:tobias.muehlenbruch@fernuni-hagen.de}{tobias.muehlenbruch@fernuni-hagen.de}}
\author{W.~Raji}
\address{\href{http://www.aub.edu.lb/fas/math/}{Department of Mathematics}, \href{http://www.aub.edu.lb/}{American University of Beirut}, Beirut, Lebanon}
\email{\href{mailto:wr07@aub.edu.lb}{wr07@aub.edu.lb}}

\date{\today}


\numberwithin{equation}{section}    

\begin{document}

\begin{abstract}
We initiate the study of generalized Maass wave forms, those Maass wave forms for which the multiplier system is not necessarily unitary.
We then prove some basic theorems inherited from the classical theory of modular forms with a generalization of some examples from the classical theory of Maass forms.
\end{abstract}


\maketitle


\section{Introduction}
\label{A}
\par The space of generalized modular forms of integer weights arise
naturally in rational conformal field theory or the theory of vertex
operator algebras \cite{DLM, Z}.  Those are meromorphic functions
defined on the upper half plane that satisfy the transformation law
under subgroups of finite index in the full modular group same as
classical modular forms with a difference that the group multiplier
appearing in the transformation law does not necessarily have
absolute value one.

\par On the other hand, Maass wave forms are real analytic functions
invariant under the action of subgroups of the full modular group,
are eigenfunctions of the Laplacian operator and at most grow like
polynomials at the cusps.  Maass wave forms connects to several
areas like $L$-series \cite{Ma83}, representation theory \cite{Bo97,
Bu97} and other connections to Artin billiard and associated
transfer operators \cite{Ma91}.

\par In this paper, we initiate the study of generalized Maass
wave forms, give basic properties and definitions and extend some
theorems from the theory of classical modular forms. We shall show
that we can construct Maass wave forms from generalized modular
forms analogous to the construction of Maass wave forms from
classical modular forms. We also construct Eisenstein series and
Poincar\'e series associated to generalized Maass wave forms.  We
continue to study the analytic properties of the mentioned forms and
in particular the Whittaker-Fourier expansion and the Maass
operators.

\par In Lemma \ref{B3.1}, we show that if we have a generalized Maass
wave form on a subgroup of the full modular group and this same form
is a classical Maass wave form of a smaller group contained in the
subgroup, then the generalized Maass wave form in classical on the
bigger group.

\par We then introduce vector valued Maass wave forms that might help in establishing the
Eichler cohomology of generalized Maass wave forms in future work.
Going to the vector valued case creates an easier tool to deal with
integral transforms associated to the periods of Eichler integrals.

\section{Generalized Maass Wave Forms}
\label{B}

\subsection{Preliminaries}
\label{B1} Let $\SL{\IR}$ denote the group of $2 \time 2$ matrices
with real entries and determinant $1$. The subgroup $\SL{\IZ}\subset
\SL{\IR}$ denotes the \emph{full modular group}, that is the
subgroup matrices with integer entries. It is generated by
\begin{equation}
\label{B1.1} S = \Matrix{0}{-1}{1}{0} \quad \text{and} \quad T =
\Matrix{1}{1}{0}{1} \qquad \text{satisfying} \quad S^2 = (ST)^3 =
-\id,
\end{equation}
where $-\id\in \SL{\IZ}$ and $\id$ denotes the identity matrix. The
group $\SL{\IR}$ acts on the upper half-plane $\IH = \{ z \in \IC
\mid \im{z} >0\}$ and its boundary $\IP_\IR = \IR \cup \{\infty\}$
by \emph{fractional linear transformations}
\begin{equation}
\label{B1.2} \Matrix{a}{b}{c}{d} \, z  :=
\begin{cases}
\frac{az+b}{cz+d} &\text{if } z \neq -\frac{d}{c}\\
\frac{a}{c} &\text{if } z = -\frac{d}{c}.
\end{cases}
\end{equation}
Moreover, we have
\begin{equation}
\label{B1.3} \im{\gamma  z} = \frac{\im{z}}{\abs{cz+d}^2} \quad
\text{and} \quad \frac{\dd}{\dd z} \gamma z = \frac{1}{(cz+d)^2}
\end{equation}
for all $\gamma= \Matrix{a}{b}{c}{d} \in \SL{\IR}$.

Let $\Gamma \subset \SL{\IZ}$ be subgroup of the full modular group
with finite index. It is known that the \emph{fundamental domain}
$\cF= \cF_\Gamma$ of $\Gamma$ in $\IH$ is a hyperbolic polygon
containing finitely many inequivalent parabolic cusps $q_1, \ldots
q_t$, $t \geq 1$. We denote the set of inequivalent cusps by
$\cC=\cC_\Gamma:= \{q_1, \ldots q_t\}$.

To each cusp $q \in \cC_\Gamma$ we denote the stabilizer of $q$ by
$\Gamma_q= \langle \gamma_q, - \id \rangle$. There exists a
\emph{scaling matrix} $g_q \in \SL{\IZ}$ such that
\begin{equation}
\label{B1.13} q = g_q \, i\infty   \qquad \text{and} \qquad
g_q^{-1} \gamma_q g_q = \Matrix{1}{l_q}{0}{1} = T^{l_q},
\end{equation}
where $l_q \in \IN$ is the \emph{width} of the cusp $q$, see
\cite[(2.1), page~40]{Iw02}, or \cite[page~5]{KM03}.

\smallskip

A \emph{multiplier} or \emph{multiplier system} $v$ compatible with
(complex) \emph{weight} $k$ is a function
\begin{equation}
\label{B1.4} v:\Gamma \to \IC_{\neq 0}
\end{equation}
such that
\begin{equation}
\label{B1.5} f(\gamma z) = v(\gamma) \,e^{ik \arg{cz+d}}\, f(z)
\end{equation}
allows non-zero solutions $f$. We call a multiplier system $v$
\emph{ weakly parabolic} if
\begin{equation}
\label{B1.9} \abs{v(\gamma)} = 1 \qquad \text{for all parabolic }
\gamma \in \Gamma,
\end{equation}
as in \cite[Equation~(5)]{KR10}.

\begin{remark}
\label{B1.11} We use the convention
\begin{equation*}
w^k =\abs{w}^k \, e^{ik \arg{w}}
\end{equation*}
with $\arg{z} \in (-\pi,\pi]$ for all $z \in \IC_{\neq 0}$ to
determine the $k^\text{th}$ power in \eqref{B1.5}.
\end{remark}

\begin{remark}
\label{B1.8} Condition~\eqref{B1.5} implies in particular that  $v$
satisfies the relation
\begin{equation}
\label{B1.6} v(\gamma \delta) \, e^{ik \arg{c_{\gamma\delta} z+
d_{\gamma \delta}}} = v(\gamma)v(\delta) \,e^{ik \arg{c_\gamma(
\delta z)+d_\gamma}} e^{ik \arg{c_\delta z+d_\delta}}
\end{equation}
for all $\gamma, \delta \in \Gamma$ and $z \in \IH$. In particular
\eqref{B1.6} implies
\begin{equation}
\label{B1.7} v(-\id) = e^{-ik} \qquad(\text{if } -\id \in \Gamma).
\end{equation}
\end{remark}

\smallskip

We also introduce the \emph{slash-action} as notation.
For $\gamma=\Matrix{a}{b}{c}{d} \in \SL{\IR}$, $k \in \IC$ and $f$
be a function on $\IH$ we define
\begin{equation}
\label{B1.10} \left(f\big|_k \gamma \right) (z) := e^{-ik
\arg{cz+d}}\, f(\gamma \,z) \qquad \text{for all $z \in \IH$}.
\end{equation}
For example Equation \eqref{B1.5} reads as $f\big|_k \gamma  =
v(\gamma) \, f$.

\subsection{Classical and Generalized Maass Wave Forms}
\label{B2}
We briefly recall Maass wave forms.
\begin{definition}
\label{B2.1} Let $\Gamma\subset \SL{\IZ}$ be subgroup with finite
index and $v: \Gamma \to \IC_{\neq 0}$ a unitary multiplier system
compatible with the real weight $k$. A \emph{classical Maass wave
form} of weight $k$, multiplier $v$ for the
group $\Gamma$ is a real-analytic function $u:\IH \to \IC$
satisfying
\begin{enumerate}
\item \label{B2.1.1}
$u\big|_k\gamma = v(\gamma) \, u$ for all $\gamma \in \Gamma$,
\item \label{B2.1.2}
$u$ is an eigenfunction of the Laplace operator $\Delta_k$ with eigenvalue
$\lambda \in \IR$, i.e., $\Delta_k u = \lambda \, u$ with $z= x+iy \in \IH$ and
\begin{equation}
\label{B2.2}
\Delta_k = -y^2 \left(\partial^2_x + \partial^2_y \right) + ik y \partial_x,
\end{equation}
\item \label{B2.1.3}
$u$ satisfies the growth condition $u(g_q\,z) = \OO{y^c}$ at each
cusp $q \in \cC_\Gamma$ as $y \to \infty$ for some $c \in \IR$ with
$g_q$ beeing the associated scaling matrix in \eqref{B1.13}.
\end{enumerate}
We denote the space of classical Maass wave form by
$\cM(\Gamma,k,v,\lambda)$.

$u$ is called a \emph{classical Maass cusp form} if $u$ satisfies
the stronger growth condition $u(g_q\,z) = \OO{y^c}$ at each cusp $q
\in \cC_\Gamma$ as $y \to \infty$ for all $c \in \IR$.
\end{definition}

Maass has proved in \cite[Theorem 28]{Ma83} that the space
$\cM(\Gamma,k,v,\lambda)$ of Maass wave forms is finite dimensional.

\medskip

Generalized Maass wave forms still keep essentially the properties
\ref{B2.1.1} and \ref{B2.1.2} of Definition~\ref{B2.1}. However we
remove the condition that the multiplier system is unitary and we
weaken the growth condition. This leads to the following
\begin{definition}
\label{B2.3} Let $\Gamma\subset \SL{\IZ}$ be subgroup with finite
index and $v: \Gamma \to \IC_{\neq 0}$ a multiplier system
compatible with the complex weight $k$.  A \emph{generalized Maass
wave form} of weight $k$, multiplier $v$ for
the group $\Gamma$ is a real-analytic function $u:\IH \to \IC$
satisfying
\begin{enumerate}
\item \label{B2.3.1}
$u\big|_k\gamma = v(\gamma) \, u$ for all $\gamma \in \Gamma$,
\item \label{B2.3.2}
$u$ is an eigenfunction of $\Delta_k$ with eigenvalue $\lambda \in \IC$, i.e., $\Delta_k u = \lambda \, u$,
%
\item \label{B2.3.3}
$u$ satisfies the growth condition $u(g_q\,z) = \OO{e^{cy}}$ in each
cusp $q \in \cC_\Gamma$ as $y \to \infty$ for some $c \in \IR$ with $g_q$ given in \eqref{B1.13}.
\end{enumerate}
We denote the space of generalized Maass wave form by
$\gM(\Gamma,k,v,\lambda)$.
\end{definition}

\subsection{A Basic Lemma}
\label{B3}

Similar to \cite[Lemma~3]{KM03} we have the following result.

\begin{lemma}
\label{B3.1} Suppose $u:\IH \to \IC$ is a classical Maass wave form
form with respect to $(\Gamma, k, v, \lambda)$, with $\Gamma$ of
finite index in $\SL{\IZ}$. (That means that $k \in \IR$ and
$\abs{v(\gamma)}=1$ for all $\gamma \in \Gamma$.\@) Suppose further
that $u$ is a generalized Maass wave form with respect to
$(\Gamma^\star, k, v^\star, \lambda)$, where $\Gamma \subset
\Gamma^\star \subset \SL{\IZ}$. (Note that $v^\star = v$ on
$\Gamma$.\@)

Then $u$ is already a classical Maass form with respect to
$(\Gamma^\star, k, v^\star, \lambda)$. That is to say: $k \in \IR$
and $\abs{ v^\star(\gamma)} = 1$ for all $\gamma \in \Gamma$ implies
that $\abs{ v^\star(\gamma)} = 1$ for all $\gamma \in \Gamma^\star$.
\end{lemma}

We adapt the proof of \cite[Lemma~3]{KM03} to our situation.

\begin{proof}
From the consistency condition \eqref{B1.6} and the fact that
$\Gamma$ has finite index in $\Gamma^\star$, it follows easily that
$v$ assumes only finitely many distinct values on $\Gamma^\star$. On
the other hand, if there were $\gamma^\star \in \Gamma^\star$ such
that $\abs{v(\gamma^\star)} \neq 1$, then by \eqref{B1.6} the set
$\left\{ \abs{v^\star\big((\gamma^\star)^n\big)}; \, n \in \IZ
\right\}$ would contain infinite many distinct values for
$\abs{v^\star}$ on $\Gamma^\star$.

We still have to check that the growth conditions of $u$ are in fact
as in Definition ~\ref{B2.1} (\ref{B2.1.3}). Since $\Gamma \subset
\Gamma^\star$ the set of cusps satisfy $\cC_{\Gamma^\star} \subset
\cC_\Gamma$. The assumptions of the lemma imply that $u$ satisfies
the stronger growth condition in Definition~\ref{B2.1} in each cusp
of $\cC_{\Gamma^\star}$. The additional cusps in $\cC_\Gamma
\smallsetminus \cC_{\Gamma^\star}$ can be transformed into a cusp
in $\cC_{\Gamma^\star}$ by an element in $\Gamma^\star$. (These
additional cusps are $\Gamma^\star$-equivalent to cusps in
$\cC_{\Gamma^\star}$.\@) Hence the the stronger growth condition of
Definition~\ref{B2.1} is also valid for these cusps.

Hence $u$ is a classical Maass wave form since it satisfies Definition~\ref{B2.1}.
\end{proof}

\subsection{Maass Operators}
\label{D1}
We denote by $\EE_k^\pm$ the differential operators
\begin{equation}
\label{D1.1} \EE^\pm_{k} = \pm 2iy\partial_x +2y\partial_y \pm k
\end{equation}
acting on real-analytic functions.
Using $\partial_z = \frac{1}{2}\partial_x - \frac{i}{2} \partial_y$ and
${\partial}_{\bar{z}} =\frac{1}{2}\partial_x + \frac{i}{2}
\partial_y$ gives
\begin{equation}
\label{D1.2} \EE^+_k = 4iy\partial_z +k \quad \text{and} \quad
\EE^-_k = -4iy\partial_{\bar{z}} -k.
\end{equation}

\begin{remark}
\label{D1.5} The Maass operators are named after Hans Maass. He
studied operators $\mathrm{K}_k$ and $\Lambda_k$, see e.g.\
\cite[page 177]{Ma83}, which are essentially $\EE_k^\pm$.
\end{remark}

As shown for example in \cite[\S6.1.4]{Br94}, the operators
$\EE^\pm_{k\mp 2} \EE^\mp_k$ are related to $\Delta_k$ by
\begin{equation}
\label{D1.4} \Delta_k = -\frac{1}{4} \EE^+_{k- 2} \EE^-_k
-\frac{k(k- 2)}{4} = -\frac{1}{4} \EE^-_{k+ 2} \EE^+_k -\frac{k(k+
2)}{4}.
\end{equation}
Direct calculations show that the slash-action commutes
with the Laplace operator
\begin{equation}
\label{D1.8} \Delta_{k}\big( u\big|_{k} \gamma \big) =
\big(\Delta_{k} u\big)\big|_{k} \gamma \qquad(\gamma \in \SL{\IR}),
\end{equation}
and interacts as follows with the Maass operators
\begin{equation}
\label{D1.9} \EE^\pm_{k}\big( u\big|_{k} \gamma \big) =
\big(\EE^\pm_{k}u\big)\big|_{k\pm 2} \gamma \qquad(\gamma \in
\SL{\IR}),
\end{equation}
for all $k\in\IC$ and smooth $u:\IH \to \IC$.

\begin{lemma}
\label{D1.6} $\EE_k^\pm$ map generalized Maass wave forms of weight
$k$ to generalized Maass forms of weight $k\pm2$:
\begin{equation}
\label{D1.3} \EE^\pm_{k}: \gM(\Gamma,k,v,\lambda) \to
\gM(\Gamma,k\pm 2,v,\lambda).
\end{equation}
\end{lemma}

\begin{proof}
Using \eqref{D1.4}, we get the commutation relation
\begin{equation}
\label{D1.7}
E^{\pm}_k \Delta_k=\Delta_{k\pm2}E^{\pm}_k.
\end{equation}
As a result, the eigenfunctions of $\Delta_k$ are mapped to the
eigenfunctions of $\Delta_{k\pm2}$ by $E^{\pm}_k$.
\eqref{D1.9} shows that the group action commute with $E^{\pm}_k$ and
\eqref{D1.1} shows that the growth condition of the generalized
Maass waveform is compatible with $E^{\pm}_k$.
\end{proof}

\section{Some Examples}
\label{C}

\subsection{Maass Wave Forms}
\label{C1} We consider \emph{Maass wave forms} as for example as
introduced in \cite{Iw02}. These are real-analytic functions
$u:\IH \to \IC$ which satisfy
\begin{enumerate}
\item $u(g \, z) = u (z)$ for all $g \in \Gamma$ and $z \in \IH$,
\item $ \Delta_0 u = \lambda u$, with non-negative real eigenvalue $\lambda$, and
\item $u(g_q\,z) = \OO{y^{c}}$ in each cusp $q \in \cC_\Gamma$ as $y= \im{z} \to \infty$ for some $c \in \IR$.
\end{enumerate}
Maass wave forms are obviously also generalized Maass wave forms for
weight $0$ and trivial multiplier.

Maass wave forms with real weight, as considered in \cite{Br94} and
\cite{Mu03}, are also generalized Maass wave forms for real weight
and unitary multiplier.

\subsection{Generalized Modular Forms}
\label{C2} Generalized modular forms are introduced a few years
back. Following \cite{KM03}, a generalized modular form $F$ is a
holomorphic function $F:\IH \to \IC$ with a left finite
Fourier-expansion
\begin{equation}
\label{C2.1} F(z) = \sum_{n=-m}^\infty a_n \, e^{2\pi i n z}
\end{equation}
at each cusp and it satisfies the transformation property
\begin{equation}
\label{C2.2} (cz+d)^{-k} \, F(\gamma z) = v(\gamma) \, F(z) \qquad
\text{for all } \gamma = \Matrix{\star}{\star}{c}{d} \in \Gamma,
\end{equation}
see \cite[\S2, Definition]{KM03}.
Taking
\begin{equation}
\label{C2.3} u(z):= \im{z}^\frac{k}{2} \, F(z),
\end{equation}
$F$ induces a generalized Maass wave form $u \in \gM\left(\Gamma,k,\rho,\frac{k}{2}\left(1-\frac{k}{2}\right) \right)$.

Indeed we have
{\allowdisplaybreaks
\begin{align*}
u(\gamma z)
&= \im{\gamma z}^\frac{k}{2} \, F(\gamma z)
 = \left( \frac{\im{z}}{\abs{cz+d}^2}\right)^\frac{k}{2} \, F(\gamma z)   \qquad \text{using \eqref{B1.3}} \\
&= \left( \frac{\im{z}}{\abs{cz+d}^2}\right)^\frac{k}{2} \, (cz+d)^k \, v(\gamma) \, F(z)    \qquad \text{using \eqref{C2.2}} \\
&= \left( \frac{(cz+d)}{\abs{cz+d}}\right)^k  \, v(\gamma) \, \im{z}^\frac{k}{2} \, F(z) \\
&= e^{ik \arg{cz+d}} \, v(\gamma) \, u(z)  \qquad \text{using $\frac{z}{\abs{z}} = e^{i \arg{z}}$}
\end{align*}
}
for each $\gamma \in \Gamma$ and
\begin{align*}
\Delta_k u (z)
&=  \left(-\frac{1}{4} \EE^+_{k-2} \EE^-_k  -\frac{k(k-2)}{4}\right) \left(\im{z}^\frac{k}{2} \, F(z)\right) \\
&=  0 - \frac{k(k-2)}{4} \left(\im{z}^\frac{k}{2} \, F(z)\right)
 = \frac{k}{2}\left(1-\frac{k}{2}\right) \, u(z)
\end{align*}
using \eqref{D1.4}, and the property that $u$ in \eqref{C2.3} lies
in the kernel of $\EE^-_k$:
\begin{align*}
\EE^-_k \, \left(\im{z}^\frac{k}{2} \, F(z)\right) &=
\bigg(-4iy\partial_{\bar{z}} -k \bigg) \left(\im{z}^\frac{k}{2} \, F(z)\right) \qquad \text{using \eqref{D1.2}}\\
&=
-4iy \partial_{\bar{z}}\left(\im{z}^\frac{k}{2} \, F(z)\right) - k \left(\im{z}^\frac{k}{2} \, F(z)\right) \\
&= - 4iy \bigg(\frac{i}{2} \partial_y \im{z}^\frac{k}{2} \bigg) \,
F(z) \;
- 4iy \im{z}^\frac{k}{2} \, \bigg(\partial_{\bar{z}}F(z)\bigg) \\
&\quad
- k \,\im{z}^\frac{k}{2} \, F(z) \\
&= 2y \, \frac{k}{2} \im{z}^{\frac{k}{2}-1} \, F(z) \; -0 \;- k \,\im{z}^\frac{k}{2} \, F(z)
 = 0.
\end{align*}
$u$ satisfies the growth property in Definition~\ref{B2.3} (\ref{B2.3.3}) since the generalized modular form $F$ has a left finite Fourier-expansion.

This generalizes the example of holomorphic modular forms in
\cite[page~6]{Mu03}.

\subsection{Eisenstein Series and Poincar\'e Series}
\label{C3} We use a method of constructing generalized Maass
waveforms similar to the classical construction as in \cite{BOR08}.
To twist the definition of the real analytic Eisenstein series and
Poincar\'e series by introducing a non-unitary multiplier systetem
inside the sum. However, this construction will definitely affect
the convergence of the series.

\begin{definition}
\label{C3.1}
Let $v$ be a multiplier system for $\Gamma$
which is compatible to weight $k$. For $F:\IH \to \IC$ an
eigenfunction of $\Delta_k$ define formally the \emph{generalized
Poincar\'e Series} $P(z)$ by the formal series
\begin{equation}
\label{mult0}
\begin{split}
P(z) &=
\sum_{\gamma \in \Gamma_\infty \backslash \Gamma}   v(\gamma)^{-1} \, \left( F\big|_k \gamma \right) (z) \\
&= \sum_{\gamma \in \Gamma_\infty \backslash \Gamma} v(\gamma)^{-1}
\, e^{-ik \arg{c_\gamma z+ d_\gamma}} \, F(\gamma z)
\end{split}
\end{equation}
where $\gamma=\begin{pmatrix} \cdot & \cdot \\ c_\gamma & d_\gamma
\end{pmatrix}$ runs to a complete set of coset representatives for
$\Gamma_\infty \backslash \Gamma$.
\end{definition}
\smallskip

It is a straight forward calculation to show that $P$ formally
satisfies properties (\ref{B2.3.1}) and (\ref{B2.3.2}) of Definition~\ref{B2.3}, provided
that the series converges absolutely. Here, $|v|$ is not necessarily $1$.

Assume for the moment that
\begin{equation}
\label{hol} F(z)= \im{z}^\frac{k}{2} \, h(z)
\end{equation}
with $h: \IH \to \IC$ a bounded holomorphic function. Recall from
\cite[Lemma~6]{Kn74} that
\begin{equation}
\label{mult1} |v(\gamma )| \leq K \,\mu(\gamma)^\alpha
\end{equation}
where $K$ is a positive constant, $\alpha$ is another constant
depending on the modulus of the multiplier system at the generators
of $\Gamma$ and $\mu(\gamma) =a^2 + b^2 + c^2 + d^2$ where $a, b, c,
d$ are the entries of $\gamma$. Recall also that there exists a
constant $K_1$ such that
\begin{equation}
\label{mult2} \mu(\gamma) \leq K_1 (c^2 + d^2 )
\end{equation}
for all $\gamma \in \Gamma_\infty \backslash \Gamma$. Moreover, we
have from \cite[Lemma~4]{Kn74} the following inequality
\begin{equation}
\label{mult3} c^2 +d^2 \leq \frac{20}{3} |cz + d|^2.
\end{equation}
Combining \eqref{mult2} and \eqref{mult3} explains the absolute
convergence of the series in \eqref{mult0} for large $k$ with $k >
2\alpha + 1$:
\begin{align*}
P(z) &=
\sum_{\gamma \in \Gamma_\infty \backslash \Gamma}   v(\gamma)^{-1} \, e^{-ik \arg{c_\gamma z+ d_\gamma}} \, \im{\gamma z}^\frac{k}{2} h(\gamma z) \\
&= \im{z}^\frac{k}{2} \sum_{\gamma \in \Gamma_\infty \backslash
\Gamma}  v(\gamma)^{-1} \, \frac{h(\gamma z)}{\mid c_\gamma z +
d_\gamma \mid^k}.
\end{align*}

\smallskip

Similarly, taking
\begin{equation}
\label{eisen} F(z)= \im{z}^{\frac{1}{2}+ \nu}
\end{equation}
defines a \emph{generalized Eisenstein series} for $\re{\nu} >
\max\{\alpha,0\}$ large enough for weight $0$ and a compatible
multiplier $v$. The absolute convergence of the series $P$ in
\eqref{mult0} follows from
\begin{align*}
P(z) &=
\sum_{\gamma \in \Gamma_\infty \backslash \Gamma}   v(\gamma)^{-1} \, \im{\gamma z}^{\frac{1}{2}+\nu} \\
&= \im{z}^{\frac{1}{2}+\nu} \sum_{\gamma \in \Gamma_\infty
\backslash \Gamma}  v(\gamma)^{-1} \, \mid c_\gamma z + d_\gamma
\mid^{-1-2\nu}.
\end{align*}

\smallskip

\begin{remark}
\label{C3.2}
Using $F(z) = \tilde{M}_{\frac{k}{2},\nu}\big(4
\pi i|n| \im{z} \big) \, e^{2 \pi in \re{z}}$ with
$\re{\nu} > \max\{\alpha,0\}$ and following the arguments in \cite[\S 5.1]{Br81} defines another type of Poincar\'e series.
\end{remark}

\section{Fourier-Whittaker Expansions}
\label{D}

\subsection{Whittaker Functions}
\label{D2} Recall Whittaker's normalized differential equation
\begin{equation}
\label{D2.1} \frac{d^2}{d y^2} G(y) + \left( -\frac{1}{4} +
\frac{k}{y} + \frac{\frac{1}{4}-\nu^2}{y^2} \right) G(y) = 0
\end{equation}
for smooth functions $G: (0,\infty) \to \IC$ and $\nu \notin -\frac{1}{2} \IN$.

According to \cite[(13.14.2), (13.14.3)]{DL13}, see also
\cite[Chapter~7]{MO66}, we have two solutions $M_{k,\nu}(y)$ and
$W_{k,\nu}(y)$ with different behavior as $y \to \infty$:
\begin{align}
\label{D2.2.1}
M_{k,\nu}(y) &\sim \frac{\Gamma(1+2\nu)}{\Gamma\left(\frac{1}{2}+\nu+k\right)} \, e^{\frac{1}{2} y} \, y^{-k} \qquad \text{and} \\
\label{D2.2.2} W_{k,\nu}(y) &\sim e^{-\frac{1}{2} y} \, y^k.
\end{align}
The asymptotic behavior is valid for $k-\nu \notin \left\{
\frac{1}{2},\frac{3}{2}, \frac{5}{2}, \ldots \right\}$, see
\cite[(13.14.20), (13.14.21)]{DL13}. These functions satisfy also
the recurrence relations \cite[(13.15.1), (13.15.11)]{DL13} and
differentiation relations \cite[(13.15.17), (13.15.20) and
(13.15.23), (13.15.26)]{DL13} and \cite[\S7.2.1, page 302]{MO66}.

We consider a modified pair of solutions.
\begin{definition}
\label{D2.5}
For $k-\nu \notin -\frac{1}{2} + \IN$ we define for all $y \in (0,\infty)$:
\begin{align}
\label{D2.5.1}
\WW{k,\nu}(y) &:= W_{k,\nu}(y) \qquad \text{and}\\
\label{D2.5.2} \MM{k,\nu}(y)
&:=\frac{\Gamma\left(\frac{1}{2}+\nu-k\right)}{\Gamma(1+2\nu)} \,
M_{k,\nu}(y).
\end{align}
\end{definition}

The definition makes sense for $\nu \in -\frac{1}{2} \IN$ since
Buchholts function
\begin{equation*}
\mathcal{M}_{k,\nu}:y \mapsto \frac{1}{\Gamma(1+2\nu)} \,
M_{k,\nu}(y), \ \ \ y >0,
\end{equation*}
remains well defined at these values of $\nu$, see \cite[\S7.1.1,
page 297]{MO66}.

\medskip
The following lemma summarizes the action of the Maass operators
$\EE^\pm_k$ on $\WW{\frac{k}{2},\nu}$:
\begin{lemma}
\label{D2.3} Let $k,\nu \in \IC$ such that $k \pm \nu \notin
\frac{1}{2}+\IZ$ and $\lambda = \frac{1}{4} -\nu^2$.
We have for $n >0$
\[
\EE^+_k \,\WW{\frac{k}{2},\nu}(4\pi ny)\,e^{2\pi inx} = -2
\WW{\frac{k+2}{2},\nu}(4\pi ny)\,e^{2\pi inx}
\]
and
\[
\EE^-_k\, \WW{ \frac{k}{2},\nu}(4\pi ny)\,e^{2\pi inx} = \left(
\frac{k(k-2)}{2}+2\lambda\right)
     \WW{\frac{k-2}{2},\nu}(4\pi ny)\,e^{2\pi inx}.
\]
For $n<0$ we have
\[
\EE^+_k \,\WW{-\frac{k}{2},\nu}(4\pi |n|y)\,e^{2\pi inx} = \left(
\frac{k(k+2)}{2} + 2\lambda \right)
     \WW{-\frac{k+2}{2},\nu}(-4\pi |n|y)\,e^{2\pi inx}
\]
and
\[
\EE^-_k \,\WW{-\frac{k}{2},\nu}(-4\pi |n| y)\,e^{2\pi inx} = -2
\WW{-\frac{k-2}{2},\nu}(-4\pi |n| y)\,e^{2\pi inx}.
\]
\end{lemma}

\begin{proof}
See \cite[proof of Lemma~4]{Mu03}, \cite[Table~4.1, page 63]{Br94}
or direct calculations using the recurrence formulas given in
\cite[\S7.2.1, page~302]{MO66}.
\end{proof}

Similar relations hold for the $\MM{k,\nu}$-function, using
formulas on \cite[page 302]{MO66}.

\begin{example}
\label{D2.4} We compute $\EE^+_k \, \MM{\frac{k}{2},\nu}(4\pi n
y)\,e^{2\pi inx}$ for $n>0$:
\begin{align}
\nonumber &\!\!\!\!
\EE^+_k \, \MM{\frac{k}{2},\nu}(4\pi ny)\,e^{2\pi inx} \\
\nonumber &=
\bigg( 2iy\partial_x + 2y\partial_y + k \bigg) \left(\frac{\Gamma\left(\frac{1-k}{2}+\nu\right)}{\Gamma(1+2\nu)} \, M_{\frac{k}{2},\nu}(4\pi ny)\,e^{2\pi inx} \right) \\
\nonumber &=
(k-4\pi ny) \, \frac{\Gamma\left(\frac{1-k}{2}+\nu\right)}{\Gamma(1+2\nu)} \, M_{\frac{k}{2},\nu}(4\pi ny)\,  e^{2\pi inx} \\
\label{D2.4.1} & \quad +
 2\, \frac{\Gamma\left(\frac{1-k}{2}+\nu\right)}{\Gamma(1+2\nu)} \,  \bigg( 4 \pi n y \,  M^\prime_{\frac{k}{2},\nu}(4\pi ny)\bigg) \,e^{2\pi inx} .
\end{align}
We use an identity from \cite[\S7.2.1, page 302]{MO66} to rewrite
$M^\prime_{\frac{k}{2},\nu}$ in \eqref{D2.4.1}:
\begin{equation*}
4 \pi n y \,  M^\prime_{\frac{k}{2},\nu}(4\pi ny)
=
\left(\frac{1}{2} + \frac{k}{2} + \nu \right)
M_{\frac{k+2}{2},\nu}(4 \pi n y) - \left( \frac{k}{2} - \frac{4 \pi
ny}{2} \right) M_{\frac{k}{2},\nu}(4 \pi n y).
\end{equation*}
Hence we find
\begin{align*}
& \EE^+_k \, \MM{\frac{k}{2},\nu}(4\pi ny)\,e^{2\pi inx}=
(k-4\pi ny) \, \frac{\Gamma\left(\frac{1-k}{2}+\nu\right)}{\Gamma(1+2\nu)} \, M_{\frac{k}{2},\nu}(4\pi ny)\,  e^{2\pi inx} \\
& \quad -
2 \left(\frac{1+k}{2}+\nu \right)\left(-\frac{1+k}{2}+\nu \right)\frac{\Gamma\left(\frac{1-(k+2)}{2}+\nu\right)}{\Gamma(1+2\nu)} \, M_{\frac{k+2}{2},\nu}(4 \pi n y) \\
& \quad -
  \big( k - 4 \pi ny \big)  \, \frac{\Gamma\left(\frac{1-k}{2}+\nu\right)}{\Gamma(1+2\nu)} \, M_{\frac{k}{2},\nu}(4 \pi n y) \,e^{2\pi inx}\\
&=
-2  \left(\frac{1+k}{2}+\nu \right)\left(-\frac{1+k}{2}+\nu \right)\frac{\Gamma\left(\frac{1-(k+2)}{2}+\nu\right)}{\Gamma(1+2\nu)} \, M_{\frac{k+2}{2},\nu}(4 \pi n y) \\
&= \left(\frac{k(k+2)}{2}-2\lambda \right) \,
\MM{\frac{k+2}{2},\nu}(4\pi ny)\,e^{2\pi inx}.
\end{align*}
\end{example}

\begin{lemma}
\label{D2.6} Let $k,\nu \in \IC$ such that $k \pm \nu \notin
\frac{1}{2}+\IZ$ and $\lambda = \frac{1}{4} -\nu^2$. We have for $n
>0$
\[
\EE^+_k \,\MM{\frac{k}{2},\nu}(4\pi ny)\,e^{2\pi inx} =
\left(\frac{k(k+2)}{2}-2\lambda \right)\MM{\frac{k+2}{2},\nu}(4\pi
ny)\,e^{2\pi inx}
\]
and
\[
\EE^-_k\, \MM{ \frac{k}{2},\nu}(4\pi ny)\,e^{2\pi inx} = 2
     \MM{\frac{k-2}{2},\nu}(4\pi ny)\,e^{2\pi inx}.
\]
For $n<0$ we have
\[
\EE^+_k \,\MM{-\frac{k}{2},\nu}(4\pi |n|y)\,e^{2\pi inx} = 2
     \MM{-\frac{k+2}{2},\nu}(-4\pi |n|y)\,e^{2\pi inx}
\]
and
\[
\EE^-_k \,\MM{-\frac{k}{2},\nu}(-4\pi |n| y)\,e^{2\pi inx} =
2\left(\frac{k(k-2)}{4}-\lambda\right) \MM{-\frac{k-2}{2},\nu}(-4\pi
|n| y)\,e^{2\pi inx}.
\]
\end{lemma}

\begin{proof}
Direct calculations using identities from \cite[\S7.2.1, page
302]{MO66} to rewrite $M^\prime_{\frac{k}{2},\nu}$ and the
functional equation $x \Gamma(x) = \Gamma(x+1)$ of the
Gamma-function.
\end{proof}

\subsection{Whittaker-Fourier Expansions}
\label{D3}

We consider a generalized Maass wave form $u \in \gM(\Gamma, k, v,
\lambda)$ with weakly parabolic multiplier $v$ and assume that we
have $\lambda = \frac{1}{4}-\nu^2$ for some $\nu \in \IC
\smallsetminus -\frac{1}{2}\IN$.

Let $q \in \cC_\Gamma$ be a cusp of width $l_q$ and $g_q \in
\SL{\IR}$ be the associated scattering matrix, see \eqref{B1.13}.
The action of the stabilizer $\gamma_q= g_q T^{l_q} g_q^{-1}$
implies that
\[
u_q := u \big|_k g_q
\]
is nearly periodic, i.e., $u_q  \big|_k T^{l_q} = v(\gamma_q) \, u_q$.
We expect an expansion of the form
\begin{equation}
\label{D3.1} u_q(x+iy) = \sum_{n \equiv \kappa \bmod{1}} a_n(y) \,
e^{\frac{2 \pi inx}{l_q}}
\end{equation}
at the cusp $i\infty$, where $\kappa \in \IR$ is given by
$v\big(\gamma_q\big) = e^{2\pi i \kappa}$. ($\kappa$ is real since
$v$ is a weakly parabolic multiplier).

The coefficients $a_n(y)$ still depend on $y$. Since $u_q$ solves
the partial differential equation
\begin{equation*}
(\Delta_k-\lambda)u=0
\end{equation*}
if and only if
\begin{equation}
\label{D3.2} \left(-y^2\partial_x^2 -y^2\partial_y^2 + iky\partial_x
- \left(\frac{1}{4}-\nu^2 \right) \right) u(x+iy) =0,
\end{equation}
we find by separation of variables that $a_n(y)=h\left(\frac{4 \pi
\eps n}{l_q}y\right)$, $\eps = \sign{n}$ and $n \neq 0$, solves the
ordinary differential equation
\begin{equation}
\label{D3.3} h^{\prime\prime}(t) + \left( - \frac{1}{4} +
\frac{1}{2} \eps k \frac{1}{t} + \frac{\frac{1}{4}-\nu^2}{t^2}
\right) h(t) =0
\end{equation}
which is the Whittaker differential equation. Solutions are the
$\WW{}$- and $\MM{}$-Whittaker functions
\begin{equation}
\label{D3.5} t \mapsto \WW{\eps \frac{k}{2},\nu}(t) \quad \text{and}
\quad t \mapsto \MM{\eps \frac{k}{2},\nu}(t).
\end{equation}
In the case $n=0$ separation of variables shows that $a_0(y)$ solves
the ordinary differential equation
\begin{equation}
\label{D3.9} t^2 \,h^{\prime\prime}(t) +
\left(\frac{1}{4}-\nu^2\right) h(t) =0;
\end{equation}
its independent solutions are
\begin{equation}
\label{D3.10} t \mapsto t^{\frac{1}{2}+\nu} \quad \text{and} \quad t
\mapsto t^{\frac{1}{2}-\nu}.
\end{equation}

Similar to the growth condition in Definition~\ref{B2.3} of
generalized Maass wave forms we assume the boundary condition
\begin{equation}
\label{D3.4} a_n(y) = \OO{e^{My}} \qquad \text{as $y \to \infty$}
\end{equation}
for some $M \in \IR$. Hence, has a Fourier-Whittaker expansion
\begin{equation}
\label{D3.6}
\begin{split}
u_q(x+iy) &= \sum_{\substack{n \equiv \kappa \bmod{1} \\ n \neq 0}}
 A_n \, \abs{n}^{-\frac{1}{2}} \, \WW{\sign{n} \frac{k}{2},\nu}\left(\frac{4\pi \abs{n}}{l_q}y\right) \, e^{\frac{2 \pi  i n}{l_q} x}\\
& \quad + \sum_{\substack{n \equiv \kappa \bmod{1} \\ n \neq 0,  \frac{2\pi \abs{n}}{l_q}< M}}
 B_n \, \abs{n}^{-\frac{1}{2}} \, \MM{\sign{n} \frac{k}{2},\nu}\left(\frac{4\pi \abs{n}}{l_q}y\right) \, e^{\frac{2 \pi  i n}{l_q} x} \\
& \quad + C_+ \, y^{\frac{1}{2}+\nu} + C_- \, y^{\frac{1}{2}-\nu}
\end{split}
\end{equation}
for some $M \in \IR$ (which corresponds to the constant $c$ in
Definition~\ref{B2.3}). The $0^\text{th}$-term coefficients $C_+$
and $C_-$ vanish if $\kappa \notin \IZ$.

\smallskip

The calculation above shows the following
\begin{proposition}
\label{D3.11} Let $u \in \gM(\Gamma, k, v, \lambda)$ and $q$ be a
cusp in $\cC_\Gamma$. Then $u(g_q \, z)$ admits a Whittaker-Fourier
expansion of the form~\eqref{D3.6}. The $0^\text{th}$-term term
vanishes if $v(\gamma_q) \neq 1$.
\end{proposition}

\begin{remark}
\label{D3.7} Since $\tilde{W}_{0,\nu} \left(\frac{1}{2} y \right) =
\sqrt{\frac{y}{\pi}} \, K_{\nu}\left(\txtfrac{1}{2}y \right)$ the
expansion in \eqref{D3.6} leads to the usual Fourier-Bessel
expansion of classical Maass cusp forms in weight $0$:
\begin{equation}
\label{D3.8} u(x+iy) = \sqrt{y} \, \sum_{n \in \IZ_{\neq 0}} A_n \,
K_{s-\frac{1}{2}} ( 2\pi \abs{n} y) \, e^{2\pi in x}.
\end{equation}
\end{remark}
\section{Vector Valued Generalized Maass Wave Forms}
\label{E}

\subsection{Induced Representations}

Let $G$ be a group and $H$ be a subgroup of $G$ of finite index
$\mu=[G:H]$. For each representation $\chi:H \to \mathrm{End}(V)$ we
consider the induced representation $\chi_H:G \to
\mathrm{End}(V_G)$, where
\[
V_G:=\{f:G\to V;\, f(hg)=\chi(h)f(g) \quad \mbox{for all }g\in G, h
\in H\}
\]
and
\[
\big(\chi_H(g)f\big)(g^\prime) = f(g^\prime g) \qquad \mbox{for all
}g,g^\prime \in G.
\]
For $V=\IC$ and $\chi$ the induced representation $\chi_H$ the
\emph{right regular representation}. (If $\chi$ is the trivial
representation, $\chi(h)=1$ for all $h \in H$, then $V_G$ is the
space of left $H$-invariant functions on $G$ or, what is the same,
functions on $H \backslash G$, and the action is by right
translation in the argument.\@) One can identify $V_G$ with $V^\mu$
using a set $\{\alpha_1,\ldots,\alpha_\mu\}$ of representatives for
$H \backslash G$, i.e.,
\[
H \backslash G = \{H\alpha_1, \ldots, H\alpha_\mu\}.
\]
Then
\[
V_G \to V^\mu \quad \text{with} \quad f \mapsto
\big(f(\alpha_1),\ldots,f(\alpha_\mu)\big)
\]
is a linear isomorphism which transports $\chi_H$ to the linear
$G$-action on $V^\mu$ given by
\[
g \cdot (v_1, \ldots,v_\mu) = \big(\chi(\alpha_1 g
\alpha_{k_1}^{-1}) v_{k_1}, \ldots,  \chi(\alpha_\mu g
\alpha_{k_\mu}^{-1}) v_{k_\mu}\big)
\]
where $k_j\in \{1,\ldots,\mu\}$ is the unique index such that
$H\alpha_j g =H\alpha_{k_j}$. To see this, one simply calculates
\[
\big( \chi_H(g)f \big)(\alpha_j) = f(\alpha_j g) = f(\alpha_j g
\alpha_{k_j}^{-1} \alpha_{k_j}) = \chi(\alpha_j g \alpha_{k_j}^{-1})
\big(f(\alpha_{k_j})\big).
\]

In the case of the right regular representation the identification
$V_G \cong \IC^\mu$ gives a matrix realization
\[
\chi_H(g) = \big( \tilde{\chi}(\alpha_i g \alpha_j^{-1}) \big)_{1
\leq i,j \leq \mu}
\]
where $\tilde{\chi}(g)=\chi(g)$ if $g \in H$ and $\tilde{\chi}(g)=0$
otherwise.

\medskip

We come back to the present situation. Take $G=\SL{\IZ}$,
$H=\Gamma$, $\mu=[\SL{\IZ}:\Gamma]$ and $g_1, \ldots, g_\mu\in
\SL{\IZ}$ as representatives of the $\Gamma$ orbits in $\SL{\IZ}$
(corresponding to the $g_k$'s in \eqref{B1.13}). We start with the
trivial character $\chi_0$ of $\Gamma$, defined as $\chi_0(h) = 1$
if $h \in \Gamma$ and $\chi_0(h)=0$ if $h \notin \Gamma$. Its right
regular representation is
\begin{equation}
\label{E1.1} \chi_0(h) := \chi_\Gamma(h) = \Big(  \delta_\Gamma
(g_i\,h\, g_j^{-1})  \Big)_{1 \leq i,j \leq \mu} \qquad \text{for
all }h \in \SL{\IZ}.
\end{equation}
$\chi_0(h)$ is is a permutation matrix for each $h$.

\smallskip

Let's extend the matrix representation even more.
\begin{definition}
\label{E1.5} We define the \emph{weight matrix of dimension $p$}
\[
w:\SL{\IZ} \times \IH \to \text{GL}(p,\IC)
\]
as a $p\times p$ matrix with complex entries satisfying
\begin{equation}
\label{E1.4} w(gh,z) = w(g,hz) \, w(h,z)
\end{equation}
for all $z\in \IH$ and  $g,h \in \SL{\IZ}$.
\end{definition}

\begin{example}
\label{E1.8}
\begin{enumerate}
\item
The right regular representation $w(h,z) = \chi_0(h)$ of the trivial
character in \eqref{E1.1} is a weight matrix of dimension $\mu$.
\item
The scalar function $w(h,z) = v(h) e^{ik \arg {c_hz+d_h}}$ for $v$ a
multiplier with weight $k$ for $\SL{\IZ}$ is a $1$-dimensional
weight matrix.
\end{enumerate}
\end{example}

\begin{lemma}
\label{E1.6} A multiplier $v$ of $\Gamma$ with weight $k$ induces a
weight matrix $w_{k,v}$ of dimension $\mu=[\SL{\IZ}:\Gamma]$ by
\begin{equation}
\label{E1.2} w_{k,v}(h,z) = \left( w_{i,j}(h,z)  \right)_{1\leq i,j
\leq \mu} \qquad (\text{for all } h \in \SL{\IZ}, z \in \IH)
\end{equation}
with
\begin{equation}
\label{E1.3}
\begin{split}
w_{i,j}(h,z) :&=
\begin{cases}
v\big(g_i\, h \, g_j^{-1}\big) \, e^{ik\arg{c_{g_i\, h \, g_j} z +d_{g_i\, h \, g_j}}} & \text{if } g_i\, h \, g_j^{-1} \in \Gamma \text{ and}\\
0 & \text{if } g_i\,h \, g_j^{-1} \not\in \Gamma
\end{cases}\\
&= \delta_\Gamma\big(g_i\, h \, g_j^{-1}\big)\, v\big(g_i\, h \,
g_j^{-1}\big) \, e^{ik\arg{c_{g_i\, h \, g_j} z +d_{g_i\, h \,
g_j}}}
\end{split}
\end{equation}
and $\delta_\Gamma(g)=1$ ($=0)$ if $g \in \Gamma$ ($\notin \Gamma$)
and $g= \Matrix{\star}{\star}{c_g}{d_g} \in \SL{\IZ}$.
\end{lemma}

\begin{proof}
We have to verify \eqref{E1.4}. Indeed, using property \eqref{B1.6}
we find
\begin{align*}
w_{k,v}(gh,z)
&=
\bigg( w_{i,j}(gh,z) \bigg)_{i,j} \\
&=
\bigg( \delta_\Gamma\big(g_i\, gh \, g_j^{-1}\big)\, v\big(g_i\, gh \, g_j^{-1}\big) \, e^{ik\arg{c_{g_i\, gh \, g_j^{-1}} z +d_{g_i\, gh \, g_j^{-1}}}} \bigg)_{i,j} \\
&=
\bigg(\sum_k \delta_\Gamma\big(g_i g g_k^{-1}\,g_k h g_j^{-1}\big)\, v\big(g_i g g_k^{-1}\, g_k h  g_j^{-1}\big) \\
&\qquad\qquad\qquad\qquad
e^{ik\arg{c_{g_i g g_k^{-1}\, g_k h  g_j^{-1}} z +d_{g_i g g_k^{-1}\, g_k h  g_j^{-1}}}} \bigg)_{i,j} \\
&=
\bigg(\sum_k \delta_\Gamma\big(g_i g g_k^{-1}\big) \delta_\Gamma\big(g_k h g_j^{-1}\big)\, v\big(g_i g g_k^{-1}\big) v\big(g_k h  g_j^{-1}\big) \\
&\qquad\qquad
e^{ik\arg{c_{g_i g g_k^{-1}} (g_k h  g_j^{-1}z) +d_{g_i g g_k^{-1}}}} \,e^{ik\arg{c_{g_k h  g_j^{-1}} z +d_{g_k h  g_j^{-1}}}}\bigg)_{i,j} \\
&=
\bigg( \delta_\Gamma\big(g_i g g_k^{-1}\big) \, v\big(g_i g g_k^{-1}\big) \, e^{ik\arg{c_{g_i g g_k^{-1}} (g_k h  g_j^{-1}z) +d_{g_i g g_k^{-1}}}} \bigg)_{i,k} \\
&\qquad
\bigg( \delta_\Gamma\big(g_k h g_j^{-1}\big)\, v\big(g_k h  g_j^{-1}\big) \, e^{ik\arg{c_{g_k h  g_j^{-1}} z +d_{g_k h  g_j^{-1}}}} \bigg)_{k,j} \\
&= w_{k,v}(g,hz) \, w_{k,v}(h,z).
\end{align*}
\end{proof}

\subsection{Vector Valued Generalized Maass Wave Forms}
\label{E2}
Following loosely \cite{KM03a} and \cite{Mu06}, we
introduce vector-valued generalized Maass wave forms.
\begin{definition}
\label{E2.1} A \emph{vector valued generalized Maass wave-form}
(\emph{vvgMF} for short) $\vec{u}:\IH \to \IC^t$ of
dimension $t\in \IN$ for $\Gamma$, weight matrix $w(\cdot,\cdot)$
and eigenvalue $\lambda \in \IC$ is a vector valued function
$\vec{u} = (u_1, \ldots,u_t)^\mathrm{tr}$ satisfying
\begin{enumerate}
\item $u_j$ is real-analytic for all $j\in \{1,\ldots,t\}$,
\item $\vec{u}(g z) = w(g,z) \, \vec{u}(z)$ for all $z \in \IH$ and $g \in \SL{\IZ}$,
\item $\Delta_k u_j = \lambda \, u_j$ for all $j\in \{1,\ldots,t\}$ and
\item $u_j(z) = \OO{e^{My}}$ as $\im{z} \to \infty$ for all $j \in \{1,\ldots,t\}$ and some $C \in \IR$.
The constant $M$ does not depend on $j$.
\end{enumerate}
$\gM_\text{vv}(\Gamma,t,w,\lambda)$ denotes the space of all vector valued generalized Maass wave-forms.
\end{definition}

\begin{example}
\label{E2.4}
\begin{enumerate}
\item
Let $\vec{u}$ be a vector valued cusp form as in
\cite[Definition~3.2]{Mu06}. Then $\vec{u}$ is also a vector valued
generalized Maass wave-form, with weight matrix $w(g,z) := \rho(g)$
given in \cite[(7)]{Mu06} and dimension $t:=[\SL{\IZ}:\Gamma]$.
\item
Let $\vec{F}$ be a vector-valued modular form of weight $k$,
multiplier $v$ and $p$-dimensional complex representation $\rho$ in
$\SL{\IZ}$ given in \cite[\S1]{KM03a}. Then
\begin{equation}
\label{E2.5} \vec{u}(z):= \im{z}^\frac{k}{2} \, \vec{F}(z)
\end{equation}
defines a vector valued generalized Maass wave-form for
$\SL{\IZ}$, with dimension $p$, weight matrix $w(g,z):=
v(g)e^{ik \arg{c_gz+d_g}} \, \rho(g)$ and spectral value
$\frac{k}{2}$, similar to \S\ref{C2}.
\end{enumerate}
\end{example}

\medskip

To each $u \in \gM(\Gamma,k,v,\lambda)$ we associate the vector
valued function $\Pi(u)$ given by
\begin{equation}
\label{E2.2}
\Pi:
\gM(\Gamma,k,v,\lambda) \to \gM_\text{vv}(\Gamma,\mu,w_{k,v},\lambda);
\quad
u \mapsto \Pi(u):= \big( u\big|_kg_1, \ldots, u\big|_kg_\mu \big)^\mathrm{tr}
\end{equation}
where $\mu=[\SL{\IZ}:\Gamma]$ denotes the index of $\Gamma$. The
function $\Pi(u)$ satisfies all four properties of a vector valued
generalized Maass wave-form in Definition~\ref{E2.1}.

Indeed, take an $u \in \gM(\Gamma,k,v,\lambda)$ and an $i \in
\{1,\ldots,\mu\}$. Obviously, $\big[\Pi(u)\big]_i = u\big|_k g_i$ is
real-analytic on $\IH$. As mentioned in \eqref{D1.8}, we have
\[
\Delta_k \big[\Pi(u)\big]_i = \Delta_k \big( u\big|_k g_i \big) =
(\Delta_k u)\big|_k g_i =\big[\Pi(\Delta_k u)\big]_i.
\]

And the growth condition for $\Pi(u)$ also follows directly from the
growth condition for $u$. To check the transformation property, take
a $h \in \SL{\IZ}$. There exists a $h^\prime \in \Gamma$ and an
unique $j \in \{1,\ldots,\mu\}$ such that $g_i h g_j^{-1} =:
h^\prime \in \Gamma$. Using the transformation property of
generalized Maass wave forms in Definition~\ref{B2.3} we find
\begin{align*}
\big[\Pi(u)\big]_i(hz)
&= u\big(g_ih\,z\big)
 = u\big(g_i h g_j^{-1}\, g_j z\big) \\
&= v\big(g_i h g_j^{-1}\big) \, e^{ik\arg{c_{g_i h g_j^{-1}} z + d_{g_i h g_j^{-1}}}} \; u (g_j z) \\
&= w_{i,j}(h,z) \, u (g_j z) \qquad (\text{with $w_{i,j}$ as in \eqref{E1.3}})\\
&= \sum_{j^\prime=1}^\mu w_{i,j^\prime}(h,z) \, u\big( g_{j^\prime} z \big)
 = \big[ w_{k,v}(h,z) \, \Pi\big(u\big) (z) \big]_i  \qquad
(\text{with \eqref{E1.2}}).
\end{align*}
This shows the transformation property
\[
\Pi\big(u\big)(hz) = w_{k,v}(h,z) \, \Pi\big(u\big) (z)
\]
of vvgMFs. Hence $\Pi(u) \in
\gM_\text{vv}(\Gamma,\mu,w_{k,v},\lambda)$.

\smallskip

On the other hand, consider the map
\begin{equation}
\label{E2.6}
\pi:
\gM_\text{vv}(\Gamma,\mu,w_{k,v},\lambda) \to \gM(\Gamma,k,v,\lambda)
\quad
\vec{u} \mapsto \big[\vec{u}\big]_j
\end{equation}
where $j \in \{1,\ldots, \mu\}$ satisfies $g_j \in \Gamma$. The map
is well defined since $j$ is uniquely determined and the function
$u:=\pi\big(\vec{u}\big)$ is in $\gM(\Gamma,k,v,\lambda)$:

$u$ satisfies the transformation property $u\big|_k
h=v(h) \,u$ for all $h \in \Gamma$ since
\begin{align*}
u(h z)
&= \big[\vec{u}(hz)\big]_j
 = \big[w_{k,v}(h,z)\,\vec{u}(z)\big]_j    \qquad (\text{transformation property of $\vec{u}$})\\
&= \sum_{l=1}^\mu w_{j,l}(h,z) \, \big[\vec{u}(z)\big]_l  \qquad (\text{using \eqref{E1.2}})\\
&= w_{j,j}(h,z) \, \big[\vec{u}(z)\big]_j  \qquad (\text{using \eqref{E1.3}})\\
&= v\big(g_j\, h \, g_j^{-1}\big) \, e^{ik\arg{c_{g_i\, h \, g_j} z +d_{g_i\, h \, g_j}}} \; u(z) \\
&= v(h) \, e^{ik\arg{c_h z +d_h}} \; u(z) \qquad (\text{using
\eqref{B1.6}}).
\end{align*}

Obviously, $u$ is also an eigenfunction of $\Delta_k$ with
eigenvalue $\lambda$.

To show that $u$ satisfies the required growth
condition in all cusps take a cuspidal point $q \in \cC_\Gamma$ of
$\Gamma$ and $g_q\in\SL{\IZ}$ satisfying $q=g_q \, i\infty$, as in
\eqref{B1.13}.
Similar to the calculation above, we find
\begin{align*}
\big[\vec{u}(\gamma g_q z) \big]_j
&= \big[w_{k,v}(g_q,z) \, \vec{u}(z)\big]_j
 = \sum_{l=1}^\mu w_{j,l}(g_q,z) \, \big[\vec{u}(z)\big]_l  \qquad (\text{using \eqref{E1.2}})\\
&= w_{j,q}(g_q,z) \, \big[\vec{u}(z)\big]_q  \qquad (\text{using
\eqref{E1.3}}),
\end{align*}
since $g_j\, g_q \, g_k^{-1} \notin \Gamma$ except for $k=q$. Hence
\begin{align*}
u(g_q z)
&= w_{j,q}(g_q,z) \, \big[\vec{u}(z)\big]_q
 = v(g_j)\, e^{ik \arg{c_{g_j} z +d_{g_j}}} \; \big[\vec{u}(z)\big]_q \\
&= \OO{e^{My}}  \quad \text{as $y \to \infty$ for some $M \in \IR$.}
\end{align*}

\begin{lemma}
\label{E2.3}
\begin{itemize}
\item
The maps $\Pi$ and $\pi$ are inverses of each other.
\item
The spaces $\gM_\text{vv}(\Gamma,\mu,w_{k,v},\lambda)$ and
$\gM(\Gamma,k,v,\lambda)$ are bijective.
\end{itemize}
\end{lemma}

\section{Conclusions and Outlook}
\label{F}
In this paper, we introduced generalized Maass wave forms, which
extend the generalized modular forms introduced in \cite{KM03} and,
simultaneously, Maass wave forms of real weight, as discussed in
\cite{Br94}. We also proved some related theorems and discussed the
expansions of those forms which in turns extends from the classical
theory of Maass forms.  On the other hand, several examples were
also introduced taking into account the bound for the multiplier
system.

As a next step, we like to extend the concept of Eichler integrals
leading to period polynomials \cite{E2} and period functions
\cite{LZ01, Mu03} attached to modular cusp forms and Maass cusp
forms. That is we like to generalize objects of the form
\[
(g,\gamma) \mapsto \int_{z_0}^{\gamma \, z_0}  g(z) (z -X )^{k-2}
\,d z,
\]
where $g:\IH \to \IC$ is a modular cusp form of weight $k$ and
$\gamma \in \Gamma$, to the setting of generalized Maass wave forms.
Hence we plan follow  \cite{KLR09, KR10} in our setting and we plan
as well to characterize the cohomology group associated to those
forms. In the end, we aim at constructing an Eichler-Shimura-type
map between the space of generalized Maass wave forms and the
suitable group cohomology.

It is worth mentioning that the vector valued Maass wave forms are
introduced in this paper for computational purposes in our future
work similar to the use of vector valued Maass cusp forms in
\cite{Mu06}. Also, to allow weight matrices instead of the scalar
valued multiplier systems and weight factors seems to be an
interesting generalization along \cite{KM03a}. This way, we can
easily pull back relations on forms for $\Gamma\subset \SL{\IZ}$ to
matrix valued relations on vector valued forms for $\SL{\IZ}$ as
illustrated in \cite{Mu06}.
\section{Acknowledgments} The authors would like to thank the
referee for his excellent comment about the algebraic aspect of
these forms and the importance of pursing this aspect. In addition,
the authors would like to thank the Center for Advanced Mathematical
Sciences (CAMS) at the American University of Beirut for giving us
the opportunity to meet in a workshop and discuss this work.




\bibliographystyle{amsalpha}

\end{document}